\documentclass[12pt,twoside,reqno]{amsart}
\usepackage[colorlinks=true,citecolor=blue]{hyperref}
\usepackage{mathptmx, amsmath, amssymb, amsfonts, amsthm, mathptmx, enumerate, color}
\usepackage{graphicx}
\usepackage{epstopdf}
\usepackage{multirow}
\usepackage{subfigure}
\usepackage{graphicx}

\newtheorem{theorem}{Theorem}[section]

\newtheorem{lemma}{Lemma}[section]
\newtheorem{proposition}{Proposition}[section]
\theoremstyle{definition}
\newtheorem{definition}{Definition}[section]
\newtheorem{example}{Example}[section]
\newtheorem{remark}{Remark}[section]

\numberwithin{equation}{section}

\begin{document}
\setcounter{page}{1}

\vspace*{1.0cm}
\title[A  New Penalty Dual-Primal Augmented Lagrangian Method]
{A  New Penalty Dual-Primal Augmented Lagrangian Method and Its Extensions}
\author[J. Liu, X. Ou, J. Chen ]{ Jie Liu$^{1}$,   Xiaoqing Ou$^{1,2}$,  Jiawei Chen$^{1}$}
\maketitle
\vspace*{-0.6cm}

\begin{center}
{\footnotesize {\it

$^1$School of Mathematics and Statistics, Southwest University, Chongqing 400715, China;\\
$^2$College of Management, Chongqing College of Humanities, Science \& Technology, Chongqing 401524, China
}}\end{center}

\vskip 4mm {\small\noindent {\bf Abstract.}
In this paper, we propose a penalty dual-primal augmented lagrangian method for solving convex minimization problems under linear equality or inequality constraints. The proposed method combines a novel penalty technique with updates the new iterates in a dual-primal order, and then be extended to solve multiple-block separable convex programming problems with splitting version and partial splitting version. We establish the convergence analysis for all the introduced algorithm in the lens of variational analysis. Numerical results on the basic pursuit problem and the lasso model are presented to illustrate the efficiency of the proposed method.

 }

\renewcommand{\thefootnote}{}
\footnotetext{ 
E-mail addresses: liujie66j@163.com (J. Liu), ouxiaoqing413@163.com (X. Ou), jwhen713@swu.edu.cn(J. Chen).
\par
Received May 7, 2023; Accepted xxx, xxx. 
\par  The manuscript was first finished in August 2021}

\section{Introduction}

In this paper, we focus on the following convex minimization problem subject to linear equality or inequality constraints:
\begin{eqnarray}\label{1.1}
\min \left \{\theta(x) \,| \,Ax=b\ (or \geq b),\ \ x\in\mathcal{X}\right\},
\end{eqnarray}
where $\theta:x\rightarrow \mathbb{R}$ is a closed proper convex  but not necessarily smooth function; $\mathcal{X}\subset \mathbb{R}^{n}$ is a nonempty closed convex sets; $A \in \mathbb{R}^{m\times n}$ is a given matrix and $b\in \mathbb{R}^{m}$ is a known vector. The problems (\ref{1.1}) is assumed to have solution throughout this paper.

 There are a large number of algorithms that can be used to solve problem (\ref{1.1}), where a benchmark method is the Augmented Lagrangian Method (ALM) which was proposed in \cite{1hes,2pow}. It plays a  significant role in both algorithmic design and practical application for various convex programming problems. We refer to \cite{3ber,4cha,5for,6roc,7bir,8glo,9ito} and the references therein.

In practice, for a given iterate $(x^k,\lambda^k)$, the iterative scheme of the classical ALM reads as

\begin{eqnarray}\label{1.2}
 \left\{ \begin{array}{ll}
x^{k+1} = \arg \min \left\{ \mathcal{L}_{\beta} (x,\lambda^{k}):=\theta(x)-\lambda^{\top}(Ax-b)+\frac{\beta}{2}\|Ax-b\|^{2}  \,|\,  x\in\mathcal{X} \right\},\\
\lambda^{k+1}= \lambda^{k}- \beta (Ax^{k+1}-b),
\end{array} \right.
\end{eqnarray}

where $\beta >0$ denotes the penalty parameter, $\lambda \in \mathbb{R}^{m}$ is the associated Lagrange multiplier. Hereafter, $ x $ and $\lambda$ are referred to the primal and dual variables respectively, and  $I$ and $\mathbf{0}$ are regarded as a identity matrix and a zero matrix with proper dimensions, respectively.

Ignoring some constant terms,the $x$-subproblem of (\ref{1.2}) can be rewritten as

$$x^{k+1} = \theta(x)+\frac{\beta}{2}\|Ax-b-\frac{1}{\beta}\lambda^{k}\|^{2}.$$

 It is obvious that the objective function $\theta$, the coefficient matrix A, and the set $\mathcal{X}$ are all appear at the same time, so it is still difficult to be solved if without utilizing some linearization techniques or inner solvers. Some existing algorithm can be applied to decoupled the objection function $\theta$ and coefficient matrix A, so as to alleviate the $x$-subproblem substantially, such as the linearized ALM \cite{10he} and primal-dual method \cite{11cha}. In above-mentioned algorithm, the $x$-subproblem only depends on $\theta$ and $\mathcal{X}$, and the proximity operator of the objective function $\theta$, which is defined as

 $$Prox_{\theta}^{\beta}(x):= \arg \min \left\{\theta(y)+\frac{\beta}{2} \|y-x\|^2\ |y \in\mathbb{R}^{n}\right\},\ \ \forall x \in \mathbb{R}^{n},\ \  \forall \beta >0, $$
 has a closed-form representation . But in order to ensure convergence, there is an extra restriction on step-size , i.e.,   $\sigma>\beta\|A^\top A\|$, where $\sigma>0$ and $\|A^\top A\|$ represent the spectral norm of $A^\top A$. Hence the step-size for solving (\ref{1.2}) becomes small when $\|A^\top A\|$ is too large, and the convergence may not be guaranteed consequently \cite{12he}. Recently, a balanced version of ALM was proposed in \cite{12he}, which has no limitation on step-size and takes the following iterative scheme:

\begin{eqnarray}\label{1.3}
{\rm (Balanced\  ALM)}\, \left\{ \begin{array}{ll}
x^{k+1}= \arg \min  \left\{ \theta(x)+\frac{1}{\beta}{\|x-(x^k+\frac{1}{\beta})A^\top \lambda^{k}\|}^2 \,|\,  x\in\mathcal{X}  \right\},\\
\lambda^{k+1}= \lambda^{k}- \left(\frac{1}{\beta} AA^\top +\delta I_m\right)^{-1} \left\{A(2x^{k+1}-x^{k})-b \right\},
\end{array} \right.
\end{eqnarray}

In which $\beta>0$ and $\delta>0$. It is clear that the $x$-subproblem of the balanced ALM decouples the objective function and the coefficient matrix without any extra condition. Namely, the parameter $\beta$ does not depend on $\|A^\top A\|$, and the $x$-subproblem  have a closed-form solution since its solution can be expressed as a proximal mapping. However, balanced ALM will take much time to update $\lambda^{k+1}$, and in practice it will take an inner solver to tackle the dual subproblem or use the well-known Cholesky factorization to deal with an equivalent linear equation of dual problem. In this sence, a new Penalty ALM was proposed in \cite{13bai} to solve it, which reads:

\begin{eqnarray}\label{1.4}
 \left\{ \begin{array}{ll}
x^{k+1}= \arg \min  \left\{ \theta(x)-\langle \lambda^{k}, Ax-b \rangle +\frac{\beta}{2}{\|A(x-x^k)\|}^2+\frac{1}{2}\|x-x^k\|_Q^2 \,|\,  x\in\mathcal{X}  \right\},\\
\lambda^{k+1}= \lambda^{k}-\beta [A(2x^{k+1}-x^{k})-b],
\end{array} \right.
\end{eqnarray}
where $\beta>0$, $Q\succ0$ is an arbitrarily given positive-defined matrix; the quadratic terms $\frac{\beta}{2}{\|A(x-x^k)\|}^2$ can be treated as a penalty term, while the quadratic  terms $\frac{1}{2}\|x-x^k\|_Q^2$ can be regarded as a matrix proximal term.

Both the balanced ALM and the new Penalty ALM update the new iterate by a primal-dual order. Exploiting the variational inequality structure of the  balanced ALM, a dual-primal version of the balanced ALM was proposed in \cite{14xu}. The novel proposed method generates the new iterates by a dual-primal order and enjoys the same computational difficulty with the original primal-dual balanced ALM, which reads

\begin{eqnarray}\label{1.5}
 \left\{ \begin{array}{ll}
\lambda^{k+1}= \lambda^{k}-\left(\frac{1}{\beta} AA^\top +\delta I_m\right)^{-1}(Ax^{k}-b),\\
x^{k+1}= \arg \min  \left\{ \theta(x)+\frac{\beta}{2}{\left\|x-\left\{x^{k}+\frac{1}{\beta}A^\top(2\lambda^{k+1}-\lambda^{k})\right\}\right\|}^2 \,|\,  x\in\mathcal{X}  \right\},
\end{array} \right.
\end{eqnarray}
where $\beta>0$ and $\delta>0$.

 It is clear that the original primal-dual balanced ALM also will take much time to update $\lambda^{k+1}$, and in practice it will take an inner solver  or use the well-known Cholesky factorization to deal with an equivalent linear equation of dual problem the same as the balanced ALM. So our main purpose is to alleviate the difficulty for solving dual-subproblem of the original primal-dual balanced ALM (\ref{1.5}) by utilizing the novel penalty technique proposed in \cite{13bai}. Inspired by the works \cite{13bai,14xu}, this paper propose a penalty dual-primal ALM combines a novel penalty technique with updates the new iterates in a dual-primal order, as follows:
\begin{center}
\fbox{%
\parbox{\textwidth}
{
{\bfseries Algorithm 1: the novel penalty dual-primal ALM }\\
Let $\beta>0$ and $Q\succ0$ be an arbitrarily given positive-defined matrix. Then the new iterate $(x^{k+1},\lambda^{k+1})$ is generated with $(x^k,\lambda^k)$ via the following steps:

\begin{eqnarray}\label{1.6}
 \left\{ \begin{array}{ll}
\lambda^{k+1}= \lambda^{k}-\beta(Ax^{k}-b),\\
x^{k+1}= \arg \min  \left\{ \theta(x)-\langle 2\lambda^{k+1}-\lambda^k,Ax-b \rangle +\frac{\beta}{2}{\|A(x-x^k)\|}^2+\frac{1}{2}\|x-x^k\|_Q^2 \,|\,  x\in\mathcal{X}  \right\}.
\end{array} \right.
\end{eqnarray}
}
}
\end{center}

Where the quadratic terms $\frac{\beta}{2}{\|A(x-x^k)\|}^2$ can be treated as a penalty term, while the quadratic  term $\frac{1}{2}{\|x-x^k\|_Q}^2$ can be regarded as a matrix proximal term.

 It is clear that the $x$-subproblem can be written equivalently as
 $$x^{k+1}= \arg \min  \left\{ \theta(x)-\langle 2\lambda^{k+1}-\lambda^k,Ax-b \rangle +\frac{1}{2}\|x-x^k\|_{\beta A^\top A+Q}^2 \,|\,  x\in\mathcal{X}  \right\}.$$

In particular,
\begin{enumerate}
  \item when talking $Q=\tau I-\beta A^\top A$ with $\tau >\beta \|A^\top A\|$, could convert the $x$-update to
$$x^{k+1}=\arg \min  \left\{ \theta(x) +\frac{\tau}{2}\|x-x^k-\frac{1}{\tau}A^\top(2\lambda^{k+1}-\lambda^k)\|^2 \,|\,  x\in\mathcal{X}  \right\};$$
  \item when talking $Q=\beta(\tau I- A^\top A$) with $\tau > \|A^\top A\|$, could convert the $x$-update to
$$x^{k+1}=\arg \min  \left\{ \theta(x) +\frac{\tau \beta}{2}\|x-x^k-\frac{1}{\tau}A^\top(2\lambda^{k+1}-\lambda^k)\|^2 \,|\,  x\in\mathcal{X} \right\},$$
\end{enumerate}
which have a closed-form solution by proximity operator of the objective function $\theta(x)$.

The dual update of (\ref{1.6}) is similar to \cite{11cha} and is comparatively much easier than that of the dual-primal balanced ALM. As said before, the global convergence of this penalty dual-primal ALM (\ref{1.6}), compared with some existing splitting algorithms, will no longer depend on the value of $\|A^\top A\|$. We also raise the extension versions of the proposed penalty dual-primal ALM (\ref{1.6}) to tackle the multi-block separable convex minimization problem with both linear equality and inequality constraints.

The paper is organized as follows. In Section 2, we recall some preliminaries which are used in the sequel. In Section 3, we establish the convergence analysis of the penalty dual-primal ALM. We extend the proposed method to solve the multiple-block separable convex problems and show its convergence analysis in Section 4. In Section 5, we further establish the partial splitting version and its convergence analysis. In Section 6, we  present some computational experiments. Finally, we show our conclusions in Section 7.

\section{Preliminaries}

In this section, we recall some fundamental variational inequality characterization and fundamental lemma to simplify the convergence analysis of the proposed novel penalty dual-primal ALM (\ref{1.6}).

In what follows, $\mathbb{R}^{n}$ will stand for the $n$-dimensional Euclidean space, $\langle x, y \rangle=x^\top y=\sum_{i=1}^n x_iy_i$,\ \ $\|x\|=\sqrt{\langle x, x\rangle},$ where $x,y \in \mathbb{R}^{n}$ and $\top$ stands for the transpose operation.

We first derive the optimality condition of the model (\ref{1.1}) in the lens of variational inequality (more detailed introduction refer to, e.g. \cite{15gu,16he,17he} ). The Lagrangian function of model (\ref{1.1}) is defined as
\begin{eqnarray}\label{2.1}
\mathcal{L}(x,\lambda):=\theta(x)-\lambda^{\top}(Ax-b),
\end{eqnarray}

with $\lambda \in \mathbb{R}^{m}$ is the Lagrange multiplier. To take into account both linear equality and inequality constraints in (\ref{1.1}),we define
\begin{eqnarray*}
\Omega:= \mathcal{X} \times \Lambda,\, \,  \mbox{where}\,  \Lambda:=
 \left\{ \begin{array}{ll}
\mathbb{R}^{m},\, \,  \mbox{if} \, \, Ax=b,\\
\mathbb R_+^{m},\, \, \mbox{if} \, \, Ax\geq b.
\end{array} \right.
\end{eqnarray*}

The pair $(x^\ast,\lambda^\ast)\in \Omega$ is called a saddle point of Lagrangian function \eqref{2.1} which means it is a solution point of \eqref{1.1}, if it satisfies

\begin{eqnarray*}
\begin{aligned}
\mathcal{L}_{\lambda \in \Lambda}(x^\ast,\lambda) \leq \mathcal{L}(x^\ast,\lambda^\ast) \leq \mathcal{L}_{x \in \mathcal{X}}(x,\lambda^\ast).
\end {aligned}
\end{eqnarray*}

Which can be separately rewritten as the following variational inequalities:

\begin{equation*}
\begin{aligned}
\left\{ \begin{array}{ll}
x^\ast \in \mathcal{X},\ \ \theta(x)-\theta(x^\ast)+(x-x^\ast)^\top(-A^\top\lambda^\ast)\geq 0,\ \forall x \in \mathcal{X},\\
\lambda^\ast \in \Lambda, \ \ \ (\lambda-\lambda^\ast)^\top(Ax^\ast-b)\geq 0,\ \forall \lambda \in \Lambda.
\end{array} \right.
\end {aligned}
\end{equation*}

Which can be further written as the following compact format:

\begin{equation}\label{2.2}
\begin{aligned}
\omega^\ast \in \Omega,\ \ \ \theta(x)-\theta(x^\ast)+(\omega-\omega^\ast)^\top F(\omega^\ast) \geq 0,\ \ \forall \omega \in \Omega,
\end {aligned}
\end{equation}

with
\begin{equation}\label{2.3}
\omega=\left(\begin{matrix} x\\ \lambda \end{matrix}\right),\ \ F(\omega)=\left(\begin{matrix} -A^\top \lambda\\ Ax-b \end{matrix}\right)\ \
and\ \ \Omega=\mathcal{X} \times \Lambda.
\end{equation}

Since the operator $F(\omega)$ defined in (\ref{2.3}) is affine with a skew-symmetric matrix and thus satisfies

\begin{equation}\label{2.4}
(\omega-\widetilde{\omega})^\top(F(\omega)-F(\widetilde{\omega}))\equiv 0\ \ \ \ \forall \omega, \widetilde{\omega} \in \Omega.
\end{equation}

We denote by $\Omega^\ast$ the solution set of the variational inequality (\ref{2.3},\ref{2.4}), which is also the saddle points of the Lagrangian function (\ref{2.1}) and the solution set of the model (\ref{1.1}).

The following basic lemma will be used frequently for our further discussions,which proof is elementary and can be found in, e.g. \cite{18bec}.

\begin{lemma}\label{lem2.1}
Let $x\in\mathcal{X}$ be a closed convex set, $\theta(x)$ and $f(x)$ are convex functions. If $f$ is differentiable, and the solution set of the minimization problem
$$\min\left\{\theta(x)+f(x)|\ x\in\mathcal{X}\right\},$$

is nonempty, then it holds that

$$x^\ast \in \arg \min \left\{\theta(x)+f(x)|\ x\in\mathcal{X}\right\}$$

if and only if

$$x^\ast \in\mathcal{X},\ \  \theta(x)-\theta(x^\ast)+(x-x^\ast)^\top\nabla f(x^\ast) \geq 0,\ \ \forall x \in \mathcal{X}.$$
\end{lemma}

We recall the so called quasi-Fej\'{e}r convergence theorem, which will be used in our convergence analysis.
\begin{definition}\label{plusdef1}
A sequence $ \left\{ u^k \right\} \subset \mathbb{R}^{n} $ is called quasi-Fej\'{e}r convergent to a nonempty set $ U\subset\mathbb{R}^{n}$
if, for every $u \in U $, there exists a sequence $ \left\{\epsilon_k \right\} \subset \mathbb{R}_{+} $ such that
$$\|u^{k+1}-u\|^2 \leq \|u^{k}-u\|^2+\epsilon_k, $$
with $\sum_{k=0}^\infty < \infty. $
\end{definition}

\begin{lemma}\label{pluslem1}\cite[Theorem 1]{19plus}
If $ \left\{ u^k \right\} \subset \mathbb{R}^{n} $ is a quasi-Fej\'{e}r convergent sequence to a nonempty set $ U$, then $ \left\{ u^k \right\} $
is bounded. Furthermore, if a cluster point $\bar{u}$ of $ \left\{ u^k \right\} $ belongs to $U$, then $\lim \limits_{k\to \infty} u^k=\bar{u}$.
\end{lemma}

\section{Convergence  analysis}

In this section, we establish the convergence analysis of the introduced penalty dual-primal  ALM, following the analogous analysis method in \cite{12he}.
We prove the first lemma which plays a key role in analysis of the convergence as follows.
\begin{lemma}\label{lem3.1}
Let $\left\{\omega^{k}=(x^{k},\lambda^{k})\right\}$ be the sequence generated by the penalty dual-primal balanced ALM (\ref{1.6}). Then we have

\begin{equation}\label{3.1}
\begin{aligned}
   \omega^{k+1}\in\Omega, \theta(x)-\theta(x^{k+1})+(\omega-\omega^{k+1})^\top F(\omega^{k+1}) \geq (\omega-\omega^{k+1})^\top H(\omega^k-\omega^{k+1}),\ \forall\omega \in \Omega,
\end {aligned}
\end{equation}

 where
 \begin{equation}\label{3.2}
\begin{aligned}
 H=\left(\begin{matrix}\beta A^\top A+Q & -A^\top \\ -A & \frac{1}{\beta}I \end{matrix}\right).
\end{aligned}
\end{equation}
\end{lemma}

\begin{proof}
For the $x$-subproblem in (\ref{1.6}), it follows from Lemma \ref{lem2.1} that $x^{k+1}\in \mathcal{X},$
\begin{equation*}
\begin{aligned}
 \theta(x)-\theta(x^{k+1})+(x-x^{k+1})^\top\left\{-A^\top(2\lambda^{k+1}-\lambda^{k})+(\beta A^\top A+Q)(x^{k+1}-x^{k})\right\} \geq 0, 
\end{aligned}
\end{equation*}
for all $x\in\mathcal{X}$,
which can be rewritten as
\begin{equation}\label{3.3}
\begin{aligned}
&x^{k+1}\in \mathcal{X},\ \ \theta(x)-\theta(x^{k+1})+(x-x^{k+1})^\top(-A^\top\lambda^{k+1}) \\
&\geq (x-x^{k+1})^\top\left\{-A^\top(\lambda^{k}-\lambda^{k+1})+(\beta A^\top A+Q)(x^{k}-x^{k+1})\right\}, \ \ \forall x\in\mathcal{X}.
\end{aligned}
\end{equation}

For the $\lambda$-subproblem in (\ref{1.6}), we have

$$ Ax^{k}-b+\frac{1}{\beta}(\lambda^{k+1}-\lambda^{k})=0 , $$

which implies that
\begin{equation*}
\begin{aligned}
(\lambda-\lambda^{k+1})^\top \left\{Ax^{k+1}-b-A(x^{k+1}-x^{k})+\frac{1}{\beta}(\lambda^{k+1}-\lambda^{k}) \right\} \geq 0,\ \ \ \forall \lambda \in \Lambda,
\end{aligned}
\end{equation*}

which leads to

\begin{equation}\label{3.4}
\begin{aligned}
(\lambda-\lambda^{k+1})^\top(Ax^{k+1}-b) \geq (\lambda-\lambda^{k+1})^\top \left\{-A(x^{k}-x^{k+1})+\frac{1}{\beta}(\lambda^{k}-\lambda^{k+1})\right\},\ \ \ \forall \lambda \in \Lambda.
\end{aligned}
\end{equation}

Combining(\ref{3.3}) and(\ref{3.4}), we have

\begin{equation*}
\begin{aligned}
\theta(x)-\theta(x^{k+1})+\left(\begin{matrix}x-x^{k+1}\\ \lambda-\lambda^{k+1}\end{matrix}\right)^{\top}
\left(\begin{matrix}-A^\top \lambda^{k+1}\\ Ax^{k+1}-b \end{matrix}\right) \geq
\left(\begin{matrix}x-x^{k+1}\\ \lambda-\lambda^{k+1}\end{matrix}\right)^{\top}
\left(\begin{matrix}\beta A^\top A+Q & -A^\top \\ -A & \frac{1}{\beta}I \end{matrix}\right)
\left(\begin{matrix}x^{k}-x^{k+1}\\ \lambda^{k}-\lambda^{k+1} \end{matrix}\right),
\end{aligned}
\end{equation*}
namely,
\begin{equation*}
\begin{aligned}
\theta(x)-\theta(x^{k+1})+(\omega-\omega^{k+1})^\top F(\omega^{k+1}) \geq (\omega-\omega^{k+1})^\top H (\omega^{k}-\omega^{k+1}).
\end{aligned}
\end{equation*}
\end{proof}

Convergence of the penalty dual-primal ALM depends on the positive definiteness of the matrix $H$, and following proposition gives the prove.

\begin{proposition}\label{pro3.1}
The matrix $H$ defined in (\ref{3.2}) is positive definite.
\end{proposition}

\begin{proof}
Note that
\begin{equation*}
\begin{aligned}
H&=\left(\begin{matrix}\beta A^\top A+Q & -A^\top \\ -A & \frac{1}{\beta}I \end{matrix}\right)\\
 &=\left(\begin{matrix}\beta A^\top A & -A^\top \\ -A & \frac{1}{\beta}I \end{matrix}\right)+\left(\begin{matrix}Q & 0 \\ 0 & 0 \end{matrix}\right) \\
 &=\left(\begin{matrix}-\sqrt{\beta} A^\top \\ \frac{1}{\sqrt{\beta}}I \end{matrix}\right)
 \left(\begin{matrix}-\sqrt{\beta} A & \frac{1}{\sqrt{\beta}}I \end{matrix}\right) +
 \left(\begin{matrix}Q & 0 \\ 0 & 0 \end{matrix}\right).
\end {aligned}
\end{equation*}

Then, for arbitrary $\omega=(x,\lambda)\neq 0$, we have

\begin{equation*}
\begin{aligned}
\omega^\top H \omega =\left\|\frac{1}{\sqrt{\beta}} \lambda-\sqrt{\beta}Ax \right\|^2+\|x\|_Q^2 >0,
\end {aligned}
\end{equation*}

and hence the matrix $H$ is positive definite.

\end{proof}

The following lemma is also the basis of convergence analysis of  the proposed novel penalty dual-primal ALM (\ref{1.6}).

\begin{lemma}\label{lem3.2}
Let $ \left\{\omega^{k}=(x^{k},\lambda^{k})\right\} $ be the sequence generated by the proposed penalty dual-primal balanced ALM (\ref{1.6}). Then, we obtain

\begin{eqnarray}\label{3.5}
\begin{aligned}
\theta(x)-\theta(x^{k+1})+(\omega-\omega^{k+1})^\top F(\omega) &\geq \frac{1}{2}(\|\omega-\omega^{k+1}\|_H^2-\|\omega-\omega^{k}\|_H^2)\\
&+\frac{1}{2}\|\omega^{k}-\omega^{k+1}\|_H^2, \ \forall \omega \in \Omega.
\end {aligned}
\end{eqnarray}
\end{lemma}

\begin{proof}

Since $F(\omega)$ is affine with a skew-symmetric matrix, like (\ref{2.4}) we have

\begin{equation*}
\begin{aligned}
(\omega-\omega^{k+1})F(\omega^{k+1})=(\omega-\omega^{k+1})F(\omega).
\end {aligned}
\end{equation*}

According to (\ref{3.1}), we have

\begin{equation}\label{3.6}
\begin{aligned}
\omega^{k+1}\in\Omega, \theta(x)-\theta(x^{k+1})+(\omega-\omega^{k+1})^\top F(\omega) \geq (\omega-\omega^{k+1})^\top H(\omega^k-\omega^{k+1}),\ \forall\omega \in \Omega.
\end {aligned}
\end{equation}

Applying the identity

\begin{equation*}
\begin{aligned}
(a-b)^\top H (c-d)=\frac{1}{2}\left\{\|a-d\|_H^2-\|a-c\|_H^2\right\}+\frac{1}{2}\left\{\|c-b\|_H^2-\|d-b\|_H^2\right\},
\end{aligned}
\end{equation*}
to the right-hand side of (\ref{3.6}) with $a=\omega$,\ \ $b=d=\omega^{k+1}$,\ \ $c=\omega^{k}$, then we obtain

\begin{equation}\label{3.7}
\begin{aligned}
(\omega-\omega^{k+1})^\top H(\omega^k-\omega^{k+1})
=\frac{1}{2}\left\{\|\omega-\omega^{k+1}\|_H^2-\|\omega-\omega^{k}\|_H^2\right\}+\frac{1}{2}\|\omega^{k}-\omega^{k+1}\|_H^2.
\end {aligned}
\end{equation}

Combining with (\ref{3.6}) and (\ref{3.7}), then we obtain the conclusion

\begin{equation*}
\begin{aligned}
\theta(x)-\theta(x^{k+1})+(\omega-\omega^{k+1})^\top F(\omega) & \geq \frac{1}{2}(\|\omega-\omega^{k+1}\|_H^2-\|\omega-\omega^{k}\|_H^2) \\
& +\frac{1}{2}\|\omega^{k}-\omega^{k+1}\|_H^2, \ \forall \omega \in \Omega.
\end {aligned}
\end{equation*}
\end{proof}

Combining the two lemmas given above, we can directly obtain the following key theorem.

\begin{lemma}\label{thm3.1}
Let$ \left\{\omega^{k}=(x^{k},\lambda^{k})\right\} $ be the sequence generated by the proposed penalty dual-primal balanced ALM (\ref{1.6}). Then, we have

\begin{equation}\label{3.8}
\begin{aligned}
\|\omega^{k+1}-\omega^\ast\|_H^2 \leq \|\omega^{k}-\omega^{\ast}\|_H^2-\|\omega^{k}-\omega^{k+1}\|_H^2,\ \ \forall \omega^\ast \in \Omega^\ast.
\end {aligned}
\end{equation}
\end{lemma}

\begin{proof}
Setting $\omega$ in (\ref{3.5}) as any fixed $\omega^\ast \in \Omega^\ast$, then we get
\begin{equation}\label{3.9}
\begin{aligned}
&\|\omega^{k}-\omega^{\ast}\|_H^2-\|\omega^{k+1}-\omega^\ast\|_H^2-\|\omega^{k}-\omega^{k+1}\|_H^2 \\
& \geq 2\left\{\theta(x^{k+1})-\theta(x^\ast)+(\omega^{k+1}-\omega^\ast)^\top F(\omega^\ast) \right\},\ \forall \omega^\ast \in \Omega^\ast.
\end{aligned}
\end{equation}

Since $\omega^\ast \in \Omega^\ast$ and $\omega^{k+1} \in \Omega^\ast$, according to (\ref{2.2},\ref{2.3}), the right-hand of the inequality (\ref{3.9}) is non-negative, i.e.

\begin{equation*}
\begin{aligned}
\theta(x^{k+1})-\theta(x^\ast)+(\omega^{k+1}-\omega^\ast)^\top F(\omega^\ast) \geq 0.
\end{aligned}
\end{equation*}

This leads to the assertion of this theorem immediately.

\end{proof}

\begin{theorem}\label{plusthe1}
The sequence $ \left\{\omega^{k}=(x^{k},\lambda^{k})\right\} $ generated by the proposed penalty dual-primal balanced ALM (\ref{1.6}) is quasi-Fej\'{e}r
convergent to $\Omega^\ast$.
\end{theorem}

\begin{proof}
It follows from (\ref{3.8}) that
\begin{equation}\label{plus1}
\begin{aligned}
\|\omega^{k+1}-\omega^\ast\|_H^2 \leq \|\omega^{k}-\omega^{\ast}\|_H^2,\ \ \forall \omega^\ast \in \Omega^\ast,
\end {aligned}
\end{equation}
by considering (\ref{plus1}) and Definition \ref{plusdef1} with $\varepsilon_k =0 ,\ \ \ k=1, 2, \ldots $, we conclude the proof.
\end{proof}

Base on the establishment of the key contraction in Theorem (\ref{plusthe1}), we can easily prove the convergence of the proposed novel penalty dual-primal  ALM (\ref{1.6}) through the following theorem.

\begin{theorem}\label{thm3.2}
Let$ \left\{\omega^{k}=(x^{k},\lambda^{k})\right\} $ be the sequence generated by the proposed penalty dual-primal balanced ALM (\ref{1.6}) and $H$ be defined in (\ref{3.2}). Then, the sequence $\left\{\omega^{k}\right\} $ converges to some $\bar{\omega} \in \Omega^\ast$.
\end{theorem}

\begin{proof}
According to Lemma \ref{pluslem1}, we know that the sequence $\left\{\omega^{k}\right\} $ is bounded. According to (\ref{3.8}), we have

\begin{equation}\label{3.10}
\begin{aligned}
\lim_{k\to \infty}\|\omega^{k}-\omega^{k+1}\|_H^2=0
\end{aligned}
\end{equation}

Let $\bar{\omega}$ be a cluster point of $\left\{\omega^{k}\right\} $, and $\left\{\omega^{k_j}\right\} $ be a subsequence converging to $\bar{\omega}$.
It follows from (\ref{3.1}) that

\begin{equation*}
\begin{aligned}
\omega^{k_j} \in \Omega, \theta(x)-\theta(x^{k_j})+(\omega-\omega^{k_j})^\top F(\omega^{k_j}) \geq (\omega-\omega^{k_j})^\top H(\omega^{k_j-1}-\omega^{k_j}),\ \forall\omega \in \Omega.
\end{aligned}
\end{equation*}

Since the matrix $H$ is positive definite, it follows from (\ref{3.10}) and the continuity of $\theta(x)$ and $F(\omega)$ that

\begin{equation*}
\begin{aligned}
\bar{\omega} \in \Omega, \ \ \ \theta(x)-\theta(\bar{x})+(\omega-\bar{\omega})^\top F(\bar{\omega}) \geq 0, \ \ \ \forall \omega \in \Omega.
\end{aligned}
\end{equation*}

This implies $ \bar{\omega} \in \Omega^\ast$. Then it follows from Lemma \ref{pluslem1} that $\lim \limits_{k\to \infty}\omega^{k}=\bar{\omega} \in \Omega^\ast $, then the proof is complete.

\end{proof}

\begin{theorem}\cite[Theorem 3.5]{12he}
Let$ \left\{\omega^{k}=(x^{k},\lambda^{k})\right\} $ be the sequence generated by the proposed penalty dual-primal balanced ALM (\ref{1.6}) and $H$ be defined in (\ref{3.2}). For any integer number $t>O$, if we define

\begin{eqnarray}\label{3.12}
\tilde{\omega}_t :=\frac{1}{t+1}\sum_{k=0}^t \omega^{k+1},
\end{eqnarray}
then we have

$$\tilde{\omega}_t \ \  \in \Omega, \theta(\tilde{x_t})-\theta(x)+(\tilde{\omega_t}-\omega)^\top F(\omega) \leq
\frac{1}{2(t+1)} \|\omega-\omega^0\|_H^2, \ \ \forall \omega \in \Omega. $$
\end{theorem}

\begin{proof}
It follows from (\ref{3.5}) that $\omega^{k+1} \in \Omega$,
\begin{eqnarray}\label{3.13}
\begin{aligned}
  \theta(x)-\theta(x^{k+1})+(\omega-\omega^{k+1})^\top F(\omega)+\frac{1}{2}\|\omega-\omega^{k}\|_H^2 \geq \frac{1}{2}\|\omega-\omega^{k+1}\|_H^2, \ \forall \omega \in \Omega.
\end {aligned}
\end{eqnarray}

Summarizing the inequalities (\ref{3.13}) over $k=0,1,\cdots,t$, we obtain

\begin{eqnarray*}
\begin{aligned}
(t+1)\theta(x)-\sum_{k=0}^t\theta(x^{k+1})+((t+1)\omega-\sum_{k=0}^t\omega^{k+1})^\top F(\omega)+\frac{1}{2}\|\omega-\omega^{0}\|_H^2 \geq 0, \ \forall \omega \in \Omega.
\end {aligned}
\end{eqnarray*}

It follows from (\ref{3.12}) that

\begin{eqnarray}\label{3.14}
\begin{aligned}
\frac{1}{t+1}\sum_{k=0}^t\theta(x^{k+1})-\theta(x)+(\tilde{\omega}_t-\omega)^\top F(\omega) \leq \frac{1}{2(t+1)}\|\omega-\omega^{0}\|_H^2, \ \forall \omega \in \Omega.
\end {aligned}
\end{eqnarray}

Note that $\tilde{\omega}_t$ define in (\ref{3.12}) is a convex combination of all iterates $\omega^k$ for $k=0,\cdots,t$, and $\theta(x)$ is convex. We thus have

$$\tilde{x}_t=\frac{1}{t+1}\sum_{k=0}^t x^{k+1},$$

and also

$$\theta(\tilde{x}_t) \leq \frac{1}{t+1}\sum_{k=0}^t \theta(x^{k+1}).$$

Substituting it into (\ref{3.14}), the theorem follows directly.
\end{proof}

The above theorem shows the worst-case $O(1/t)$ convergence rate of the proposed penalty dual-primal balanced ALM \eqref{1.6}, where $t$ denotes the total iteration counter.

\section{Splitting version}

In this section, following the same extension technique in \cite{12he}, we also design the splitting version for the dual-primal balanced ALM (\ref{1.6}) to solve the following multi-block separable convex optimization problem with both linear equality and inequality constraints:

\begin{equation}\label{4.1}
\min\left\{\sum_{i=1}^{p}\theta_i(x_i)|\ \sum_{i=1}^{p}A_{i}x_{i}=b\ (or \geq b),\ x_{i}\in\mathcal{X}_{i}\right\},
\end{equation}
where $p\geq 1$ is the number of subfunctions, and $\sum_{i=1}^{p}n_{i}=n$; $\theta_{i},\ i=1, \ldots, p,$ are closed proper convex  but not necessarily smooth functions; $\mathcal{X}_{i}\subset \mathbb{R}^{n_{i}},\ i=1, \ldots, p,$ are nonempty closed convex sets; $A_{i} \in \mathbb{R}^{m\times n_{i}},\ i=1, \ldots, p,$ are given matrixs and $b\in \mathbb{R}^{m}$ is a known vector. The model (\ref{4.1}) can be applied to various domains, such as, e.g. \cite{20boy,21sun,22yua}. It is clear that the model (\ref{4.1}) is an extension of the model (\ref{1.1}) from $p=1$ to $p\geq 1$.

\subsection{Splitting version of the penalty dual-primal ALM}

In this section, we extend the penalty dual-primal ALM to solve the multi-block separable convex optimization problem (\ref{4.1}) and proposed a splitting version of (\ref{1.6}) below.

\begin{center}
\fbox{%
\parbox{\textwidth}
{
{\bfseries Algorithm 2: the spitting penalty dual-primal balanced ALM }\\
Let $i=1,\cdots, p,$ and $x_{i}\in\mathcal{X}_{i}$, $\beta_{i}>0$; $Q_{i}\succ0$ are arbitrarily given positive-defined matrixes. Then the new iterate $(x^{k+1},\lambda^{k+1})$ is generated with $(x^k,\lambda^k)$ via the following steps:

\begin{eqnarray}\label{4.2}
 \left\{ \begin{array}{ll}
\lambda^{k+1}= \lambda^{k}-\sum_{i=1}^{p}\beta_{i}(\sum_{i=1}^{p}A_{i}x_{i}^{k}-b),\\
x_{i}^{k+1}= \arg \min  \left\{ \theta(x_{i})-\langle 2\lambda^{k+1}-\lambda^k,A_{i}x_{i}-b \rangle +\frac{\beta_{i}}{2}{\|A_{i}(x_{i}-x_{i}^k)\|}^2+\frac{1}{2}\|x_{i}-x_{i}^k\|_{Q_{i}}^2   \right\}.
\end{array} \right.
\end{eqnarray}
}
}
\end{center}

Where the quadratic terms $\frac{\beta_{i}}{2}{\|A_{i}(x_{i}-x_{i}^k)\|}^2$ can be treated as the penalty terms, while the quadratic  terms $\frac{1}{2}\|x-x^k\|_Q^2$ can be regarded as the matrix proximal terms.

\subsection{Variational inequality characterization of the splitting version }

In order to analyze convergence, we also need to give the optimality condition of the model (\ref{4.2}) in the variational inequality context. Firstly, we reuse the letters and set

\begin{equation*}
\Omega:= \mathcal{X}_{1} \times \ldots \times \mathcal{X}_{p} \times \Lambda\ \  where\ \ \Lambda:=
 \left\{ \begin{array}{ll}
\mathbb{R}^{m},\ \ \ if \ \ \sum_{i=1}^{p}A_{i}x_{i}=b,\\
\mathbb R_+^{m},\ \ \ if \ \ \sum_{i=1}^{p}A_{i}x_{i}\geq b.
\end{array} \right.
\end{equation*}

Analogous to the analysis in Section 2, it is clear that the optimal condition of (\ref{4.2}) is equivalent to finding a saddle point of the Lagrangian function of model (\ref{4.2}), which satisfy

\begin{equation}\label{4.3}
\begin{aligned}
\omega^\ast \in \Omega,\ \ \ \theta(x)-\theta(x^\ast)+(\omega-\omega^\ast)^\top F(\omega^\ast) \geq 0,\ \ \forall \omega \in \Omega,
\end {aligned}
\end{equation}

where
\begin{equation}\label{4.4}
\theta=\sum_{i=1}^{p}\theta_{i},\ \ \omega=\left(\begin{matrix} x\\ \lambda \end{matrix}\right),\ \ x=\left(\begin{matrix} x_{1}\\ \vdots\\ x_{p}\end{matrix}\right),\ \
F(\omega)=\left(\begin{matrix} -A_{1}^\top \lambda\\ \vdots\\ -A_{p}^\top \lambda\\ \sum_{i=1}^{p}A_{i}x_{i}-b \end{matrix}\right)\ \
and\ \ \Omega=\mathcal{X} \times \Lambda.
\end{equation}

Similarly, since the operator $F(\omega)$ defined in (\ref{4.4}) is affine with a skew-symmetric matrix and thus satisfies

\begin{equation}\label{4.5}
(\omega-\widetilde{\omega})^\top(F(\omega)-F(\widetilde{\omega}))\equiv 0\ \ \ \forall \omega, \widetilde{\omega} \in \Omega.
\end{equation}

\subsection{Convergence analysis of the splitting version}

According to the same analysis route in Section 3, we next give the essential lemma and the key proposition below.

\begin{lemma}\label{lem4.1}
Let $\left\{\omega^{k}=(x^{k},\lambda^{k})\right\}$ be the sequence generated by the spitting penalty dual-primal balanced ALM (\ref{4.2}). Then we have
\begin{equation}\label{4.6}
\begin{aligned}
   \omega^{k+1}\in\Omega, \theta(x)-\theta(x^{k+1})+(\omega-\omega^{k+1})^\top F(\omega^{k+1}) \geq (\omega-\omega^{k+1})^\top H(\omega^k-\omega^{k+1}),\ \forall\omega \in \Omega,
\end {aligned}
\end{equation}
 where
 \begin{equation}\label{4.7}
\begin{aligned}
 H=\left(\begin{matrix}\beta_1 A_1^\top {A_1}+Q_1 & \ldots & \mathbf{0} & -A_1^\top \\ \vdots  & \ddots & \vdots& \vdots\\
                        \mathbf{0} & \ldots & \beta_P A_p^\top {A_P}+Q_P & -A_P^\top \\ -A_1 & \ldots & -A_P & \sum_{i=1}^p \frac{1}{\beta_i}I
                        \end{matrix}\right).
\end{aligned}
\end{equation}
\end{lemma}

\begin{proof}
According to Lemma \ref{lem2.1}, for $i=1,\ldots, p, x_i^{k+1}\in \mathcal{X}_i$, 

\begin{equation*}
\begin{aligned}
 \theta_i(x_i)-\theta_i(x_i^{k+1})+(x-x_i^{k+1})^\top\left\{-A_i^\top(2\lambda^{k+1}-\lambda^{k})+
(\beta_i A_i^\top A_i+Q_i)(x_i^{k+1}-x_i^{k})\right\} \geq 0,
\end{aligned}
\end{equation*}
 for all ${x_i} \in\mathcal{X}_i$.
It can be written as
\begin{equation}\label{4.8}
\begin{aligned}
&x_i^{k+1}\in \mathcal{X}_i,\ \ \theta({x_i})-\theta(x_i^{k+1})+({x_i}-x_i^{k+1})^\top(-A_i^\top\lambda^{k+1})\\
& \geq ({x_i}-x_i^{k+1})^\top\left\{-A_i^\top(\lambda^{k}-\lambda^{k+1})+(\beta_i A_i^\top A_i+Q_i)(x_i^{k}-x_i^{k+1})\right\}, \ \ \forall {x_i}\in\mathcal{X}_i.
\end{aligned}
\end{equation}

For the $\lambda$-subproblem in (\ref{4.2}), we have
$$ \sum_{i=1}^{p}A_ix_i^{k}-b+\sum_{i=1}^{p}\frac{1}{\beta_i}(\lambda^{k+1}-\lambda^{k})=0,$$
which implies that
\begin{equation*}
\begin{aligned}
(\lambda-\lambda^{k+1})^\top \left\{\sum_{i=1}^{p}A_ix_i^{k+1}-b-\sum_{i=1}^{p}A_i(x_i^{k+1}-x_i^{k})+\sum_{i=1}^{p}\frac{1}{\beta_i}(\lambda^{k+1}-\lambda^{k}) \right\} \geq 0,\ \ \forall \lambda \in \Lambda,
\end{aligned}
\end{equation*}

which leads to

\begin{equation}\label{4.9}
\begin{aligned}
(\lambda-\lambda^{k+1})^\top(\sum_{i=1}^{p}A_ix_i^{k+1}-b) \geq
(\lambda-\lambda^{k+1})^\top (-\sum_{i=1}^{p}A_i(x_i^{k}-x_i^{k+1})+\sum_{i=1}^{p}\frac{1}{\beta_i}(\lambda^{k}-\lambda^{k+1})),
\end{aligned}
\end{equation}
for all $\lambda \in \Lambda$.
Combining (\ref{4.8}) and (\ref{4.9}), we have
\begin{equation*}
\begin{aligned}
&\theta(x)-\theta(x^{k+1})+\left(\begin{matrix}x-x^{k+1}\\ \lambda-\lambda^{k+1}\end{matrix}\right)^{\top}
\left(\begin{matrix}-A_1^\top \lambda^{k+1}\\ \cdots \\ -A_p^\top \lambda^{k+1} \\ \sum_{i=1}^{p}A_ix_i^{k+1}-b \end{matrix}\right)\\
&\geq\left(\begin{matrix}x-x^{k+1}\\ \lambda-\lambda^{k+1}\end{matrix}\right)^{\top}
\left(\begin{matrix}\beta_1 A_1^\top {A_1}+Q_1 & \ldots & \mathbf{0} & -A_1^\top \\ \vdots  & \ddots & \vdots& \vdots\\
                        \mathbf{0} & \ldots & \beta_P A_p^\top {A_P}+Q_P & -A_P^\top \\ -A_1 & \ldots & -A_P & \sum_{i=1}^p \frac{1}{\beta_i}I
                        \end{matrix}\right)
\left(\begin{matrix}x^{k}-x^{k+1}\\ \lambda^{k}-\lambda^{k+1} \end{matrix}\right),
\end{aligned}
\end{equation*}

namely

\begin{equation*}
\begin{aligned}
\theta(x)-\theta(x^{k+1})+(\omega-\omega^{k+1})^\top F(\omega^{k+1}) \geq (\omega-\omega^{k+1})^\top H (\omega^{k}-\omega^{k+1}).
\end{aligned}
\end{equation*}
\end{proof}

\begin{proposition}\label{pro4.1}
The matrix $H$ defined in (\ref{4.7}) is positive definite.
\end{proposition}

\begin{proof}
Note that

\begin{equation*}
\begin{aligned}
H=\bar{H}+\left(\begin{matrix}Q_1 & \cdots & \mathbf{0} & \mathbf{0} \\ \vdots & \cdots & \vdots & \vdots \\
 \mathbf{0} & \ddots & Q_P & \mathbf{0} \\ \mathbf{0} & \cdots & \mathbf{0} & \mathbf{0} \end{matrix}\right)
\end {aligned}
\end{equation*}
 and
\begin{equation*}
\begin{aligned}
\bar{H}&=\left(\begin{matrix}\beta_1 A_1^\top {A_1} & \ldots & \mathbf{0} & -A_1^\top \\ \vdots  & \ddots & \vdots& \vdots\\
\mathbf{0} & \ldots & \beta_P A_p^\top {A_P} & -A_P^\top \\ -A_1 & \ldots & -A_P & \sum_{i=1}^p \frac{1}{\beta_i}I \end{matrix}\right)\\
 &=\left(\begin{matrix}\beta_1 A_1^\top {A_1} & \ldots & \mathbf{0} & -A_1^\top \\ \vdots  & \ddots & \vdots& \vdots\\
\mathbf{0} & \ldots & \mathbf{0} & \mathbf{0} \\ -A_1 & \ldots & \cdots &  \frac{1}{\beta_1}I  \end{matrix}\right)+\cdots
+\left(\begin{matrix} \mathbf{0} & \ldots & \mathbf{0} & \mathbf{0} \\ \vdots  & \ddots & \vdots& \vdots\\
\mathbf{0} & \ldots & \beta_P A_p^\top {A_P} & -A_P^\top \\ \mathbf{0} & \ldots & -A_p &  \frac{1}{\beta_p}I  \end{matrix}\right)\\
&= \left(\begin{matrix}-\sqrt{\beta_1} A_1^\top \\ \mathbf{0} \\ \vdots \\ \mathbf{0} \\ \frac{1}{\sqrt{\beta_1}}I \end{matrix}\right)
\left(\begin{matrix} -\sqrt{\beta_1} {A_1} & \mathbf{0} &\cdots & \mathbf{0} & \frac{1}{\sqrt{\beta_1}}I \end{matrix}\right)
 +\cdots+\left(\begin{matrix} \mathbf{0} \\ \vdots \\ \mathbf{0} \\ -\sqrt{\beta_P} A_P^\top \\ \frac{1}{\sqrt{\beta_P}}I \end{matrix}\right)
 \left(\begin{matrix} \mathbf{0} & \cdots & \mathbf{0} & -\sqrt{\beta_P} {A_P} & \frac{1}{\sqrt{\beta_P}}I\end{matrix}\right)
\end{aligned}
\end{equation*}
Then, for arbitrary $\omega=(x,\lambda)\neq 0$, we have

\begin{equation*}
\begin{aligned}
\omega^\top H \omega =\sum_{i=1}^p \left\|\frac{1}{\sqrt{\beta_i}} \lambda-\sqrt{\beta_i}A_ix_i \right\|^2+\sum_{i=1}^p \|x_i\|_{Q_i}^2 >0
\end {aligned}
\end{equation*}
and hence the matrix $H$ is positive definite.
\end{proof}

\begin{remark}
Based on the Lemma \ref{lem4.1} and Proposition \ref{pro4.1}, similar lemmas and theorems for convergence and convergence rate as those in Section 3 can be obtained for the splitting version of the penalty dual-primal ALM (\ref{4.2}). Hence it is easy to prove the worst-case $O(1/t)$ convergence rate of the spitting penalty dual-primal balanced ALM (\ref{4.2}), where $t$ denotes the total iteration counter. For brevity, we omit the details proof.
\end{remark}

\section{Partial proximal strategy for the penalty dual-primal ALM}

The penalty dual-primal ALM (\ref{1.6}) can be generalized to the splitting version (\ref{4.2}) when the background issue changes from the one-block (\ref{1.1}) to the multiple-block case (\ref{4.1}). We can also proposed another way for the generalization, which only adds proximal matrix terms to partial $x_i$-subproblem of the splitting version (\ref{4.2}). In section 4, it is clear that each of the $x_i$-subproblem  involve the matrix proximal terms $\frac{1}{2}\|x_i-x_i^k\|_{Q_i}^2$ , so that it has two quadratic terms in each $x_i$-subproblem. In this sense, we think about partial proximity to the $x_i$-subproblem, so we proposed the Partial proximal penalty dual-primal ALM to solve the same multi-block separable convex optimization problem as in Section 4 (\ref{4.1}).

\subsection{Partial proximal penalty dual-primal ALM}

Without ambiguity, we add the proximal matrix terms to the former $p_1$ $x_i$-subproblem, while no to the latter $p_2$ $x_i$-subproblem, where
$1\leq p_1, p_2 \leq p$, and $p_1 +p_2 =p$. For succinctness of notation, we reused the notation in Section 4.

A partial proximal penalty dual-primal ALM for the multiple-block model (\ref{4.1}) can be read as:

\begin{center}
\fbox{%
\parbox{\textwidth}
{
{\bfseries Algorithm 3: the partial proximal penalty dual-primal ALM }\\
Let$\beta_{i}>0$ and $Q_{i}\succ0$ are arbitrarily given positive-defined matrixes, $i=1,\cdots, p,$ and $x_{i}\in\mathcal{X}_{i}$.  Then the new iterate $(x^{k+1},\lambda^{k+1})$ is generated with $(x^k,\lambda^k)$ via the following steps:

\begin{eqnarray}\label{5.1}
 \left\{ \begin{array}{ll}
\lambda^{k+1}= \lambda^{k}-\sum_{i=1}^{p}\beta_{i}(\sum_{i=1}^{p}A_{i}x_{i}^{k}-b),\\
x_{i}^{k+1}= \arg \min  \left\{ \theta(x_{i})-\langle 2\lambda^{k+1}-\lambda^k,A_{i}x_{i}-b \rangle +\frac{1}{2}\|x_{i}-x_{i}^k\|_{\beta_i A_i^\top A_i +Q_{i}}^2  \right\}, i=1, \cdots, p_{1},\\
x_{i}^{k+1}= \arg \min  \left\{ \theta(x_{i})-\langle 2\lambda^{k+1}-\lambda^k,A_{i}x_{i}-b \rangle +\frac{\beta_{i}}{2}{\|A_{i}(x_{i}-x_{i}^k)\|}^2\right\}, i= p_{1} +1, \cdots, p.
\end{array} \right.
\end{eqnarray}
}
}
\end{center}

\subsection{Convergence analysis of the partial proximal penalty dual-primal ALM (\ref{5.1})}

According to the same variational inequality characterization and the same analysis route in Section 4, we next give the essential lemma and the key proposition below.

\begin{lemma}\label{lem5.1}
Let $\left\{\omega^{k}=(x^{k},\lambda^{k})\right\}$ be the sequence generated by the partial proximal penalty dual-primal ALM (\ref{5.1}). Then we have
\begin{equation}\label{5.2}
\begin{aligned}
   \omega^{k+1}\in\Omega, \theta(x)-\theta(x^{k+1})+(\omega-\omega^{k+1})^\top F(\omega^{k+1}) \geq (\omega-\omega^{k+1})^\top H(\omega^k-\omega^{k+1}),\ \forall\omega \in \Omega,
\end {aligned}
\end{equation}
 where $H$ defined as
 \begin{equation}\label{5.3}
\begin{aligned}
 \left(\begin{matrix}
 \beta_1 A_1^\top {A_1}+Q_1 & \ldots & \mathbf{0} & \mathbf{0} & \cdots & \mathbf{0} & -A_1^\top \\
  \vdots  & \ddots & \vdots& \vdots&  \ddots & \vdots & \vdots& \\
  \mathbf{0} & \ldots & \beta_{P_1} A_{P_1}^\top {A_{P_1}}+Q_{P_1} & \mathbf{0} & \ldots & \mathbf{0} & -A_{P_1}^\top \\
  \mathbf{0}& \ldots & \mathbf{0} & \beta_{P_1+1} A_{P_1+1}^\top {A_{P_1+1}} & \ldots & \mathbf{0} & -A_{P_1+1}^\top\\
  \vdots & \ddots & \vdots & \vdots & \ddots & \vdots& \vdots \\
  \mathbf{0} & \ldots & \mathbf{0} & \mathbf{0} & \ldots & \beta_{P} A_{P}^\top {A_{P}} & -A_{P}^\top \\
  -A_1 & \ldots & -A_{P_1} & -A_{P_1+1} & \ldots & -A_{P} & \sum_{i=1}^p \frac{1}{\beta_i}I
   \end{matrix}\right).
\end{aligned}
\end{equation}
\end{lemma}

\begin{proof}
According to Lemma \ref{lem2.1}, for $i=1,\ldots, p_1, x_i^{k+1}\in \mathcal{X}_i$,

\begin{equation*}
\begin{aligned}
 \theta_i(x_i)-\theta_i(x_i^{k+1})+(x-x_i^{k+1})^\top\left\{-A_i^\top(2\lambda^{k+1}-\lambda^{k})+
(\beta_i A_i^\top A_i+Q_i)(x_i^{k+1}-x_i^{k})\right\} \geq 0, 
\end{aligned}
\end{equation*}
for all ${x_i} \in\mathcal{X}_i$.
Which can be written as
\begin{equation}\label{5.4}
\begin{aligned}
&x_i^{k+1}\in \mathcal{X}_i,\ \ \theta({x_i})-\theta(x_i^{k+1})+({x_i}-x_i^{k+1})^\top(-A_i^\top\lambda^{k+1})\\
& \geq ({x_i}-x_i^{k+1})^\top\left\{-A_i^\top(\lambda^{k}-\lambda^{k+1})+(\beta_i A_i^\top A_i+Q_i)(x_i^{k}-x_i^{k+1})\right\}, \ \ \forall {x_i}\in\mathcal{X}_i.
\end{aligned}
\end{equation}

Similarly, for $i=p_1 +1,\ldots, p$, according to Lemma \ref{lem2.1} we have $x_i^{k+1}\in \mathcal{X}_i$,
\begin{equation*}
\begin{aligned}
 \theta_i(x_i)-\theta_i(x_i^{k+1})+(x-x_i^{k+1})^\top\left\{-A_i^\top(2\lambda^{k+1}-\lambda^{k})+
\beta_i A_i^\top A_i(x_i^{k+1}-x_i^{k})\right\} \geq 0, 
\end{aligned}
\end{equation*}
for all ${x_i} \in\mathcal{X}_i$.
Which can be written as
\begin{equation}\label{5.5}
\begin{aligned}
&x_i^{k+1}\in \mathcal{X}_i,\ \ \theta({x_i})-\theta(x_i^{k+1})+({x_i}-x_i^{k+1})^\top(-A_i^\top\lambda^{k+1})\\
& \geq ({x_i}-x_i^{k+1})^\top\left\{-A_i^\top(\lambda^{k}-\lambda^{k+1})+\beta_i A_i^\top A_i (x_i^{k}-x_i^{k+1})\right\}, \ \ \forall {x_i}\in\mathcal{X}_i.
\end{aligned}
\end{equation}
For the $\lambda$-subproblem in (\ref{5.1}), we have
$$ \sum_{i=1}^{p}A_ix_i^{k}-b+\sum_{i=1}^{p}\frac{1}{\beta_i}(\lambda^{k+1}-\lambda^{k})=0,$$
which implies that
\begin{equation*}
\begin{aligned}
(\lambda-\lambda^{k+1})^\top \left\{\sum_{i=1}^{p}A_ix_i^{k+1}-b-\sum_{i=1}^{p}A_i(x_i^{k+1}-x_i^{k})+\sum_{i=1}^{p}\frac{1}{\beta_i}(\lambda^{k+1}-\lambda^{k}) \right\} \geq 0,\ \ \forall \lambda \in \Lambda,
\end{aligned}
\end{equation*}
which leads to

\begin{equation}\label{5.6}
\begin{aligned}
(\lambda-\lambda^{k+1})^\top(\sum_{i=1}^{p}A_ix_i^{k+1}-b) \geq
(\lambda-\lambda^{k+1})^\top (-\sum_{i=1}^{p}A_i(x_i^{k}-x_i^{k+1})+\sum_{i=1}^{p}\frac{1}{\beta_i}(\lambda^{k}-\lambda^{k+1})),\ \ \ \forall \lambda \in \Lambda.
\end{aligned}
\end{equation}
Combining(\ref{5.4}), (\ref{5.5}) and(\ref{5.6}), we obtain

\begin{equation*}
\begin{aligned}
\theta(x)-\theta(x^{k+1})+\left(\begin{matrix}x-x^{k+1}\\ \lambda-\lambda^{k+1}\end{matrix}\right)^{\top}
\left(\begin{matrix}-A_1^\top \lambda^{k+1}\\ \cdots \\ -A_p^\top \lambda^{k+1} \\ \sum_{i=1}^{p}A_ix_i^{k+1}-b \end{matrix}\right)
\geq\left(\begin{matrix}x-x^{k+1}\\ \lambda-\lambda^{k+1}\end{matrix}\right)^{\top} H \left(\begin{matrix}x^{k}-x^{k+1}\\ \lambda^{k}-\lambda^{k+1} \end{matrix}\right),
\end{aligned}
\end{equation*}
namely,

\begin{equation*}
\begin{aligned}
\theta(x)-\theta(x^{k+1})+(\omega-\omega^{k+1})^\top F(\omega^{k+1}) \geq (\omega-\omega^{k+1})^\top H (\omega^{k}-\omega^{k+1}).
\end{aligned}
\end{equation*}
\end{proof}

\begin{proposition}\label{pro5.1}
The matrix $H$ defined in (\ref{5.3}) is positive definite.
\end{proposition}

\begin{proof}
Note that

\begin{equation*}
\begin{aligned}
H=\bar{H}+
\left(\begin{matrix}
Q_1 & \cdots & \mathbf{0} & \mathbf{0} & \cdots & \mathbf{0} \\
 \vdots & \ddots & \vdots & \vdots & \ddots & \vdots \\
 \mathbf{0} & \cdots & Q_{P_1} & \mathbf{0} & \cdots & \mathbf{0} \\
 \mathbf{0} & \cdots & \mathbf{0} & \mathbf{0} & \cdots & \mathbf{0} \\
 \vdots & \ddots & \vdots & \vdots & \ddots & \vdots\\
 \mathbf{0} & \cdots & \mathbf{0} & \mathbf{0} & \cdots & \mathbf{0} \\
 \mathbf{0} & \cdots & \mathbf{0} & \mathbf{0} & \cdots & \mathbf{0}
 \end{matrix}\right)
\end {aligned}
\end{equation*}
 and

\begin{equation*}
\begin{aligned}
\bar{H}&=\left(\begin{matrix}
\beta_1 A_1^\top {A_1} & \ldots & \mathbf{0} & \ldots & \mathbf{0} & -A_1^\top \\
\vdots  & \ddots & \vdots & \ddots & \vdots & \vdots \\
\mathbf{0} & \ldots & \beta_{P_1} A_{P_1}^\top {A_{P_1}} & \ldots & \mathbf{0} & -A_{P_1}^\top \\
\vdots & \ddots & \vdots & \ddots & \vdots & \vdots\\
\mathbf{0} & \ldots & \mathbf{0} & \ldots & \beta_{P} A_{P}^\top {A_{P}} & -A_{P}^\top\\
-A_1 & \ldots & -A_{P_1} & \ldots & -A_{P} & \sum_{i=1}^p \frac{1}{\beta_i}I \end{matrix}\right)\\
 &=\left(\begin{matrix}\beta_1 A_1^\top {A_1} & \ldots & \mathbf{0} & -A_1^\top \\ \vdots  & \ddots & \vdots& \vdots\\
\mathbf{0} & \ldots & \mathbf{0} & \mathbf{0} \\ -A_1 & \ldots & \cdots &  \frac{1}{\beta_1}I  \end{matrix}\right)+\cdots
+\left(\begin{matrix} \mathbf{0} & \ldots & \mathbf{0} & \mathbf{0} \\ \vdots  & \ddots & \vdots& \vdots\\
\mathbf{0} & \ldots & \beta_P A_p^\top {A_P} & -A_P^\top \\ \mathbf{0} & \ldots & -A_p &  \frac{1}{\beta_p}I  \end{matrix}\right)\\
&= \left(\begin{matrix}-\sqrt{\beta_1} A_1^ \top \\ \mathbf{0} \\ \vdots \\ \mathbf{0} \\ \frac{1}{\sqrt{\beta_1}}I \end{matrix}\right)
\left(\begin{matrix} -\sqrt{\beta_1} {A_1} & \mathbf{0} &\cdots & \mathbf{0} & \frac{1}{\sqrt{\beta_1}}I \end{matrix}\right)
 +\cdots+\left(\begin{matrix} \mathbf{0} \\ \vdots \\ \mathbf{0} \\ -\sqrt{\beta_P} A_P^\top \\ \frac{1}{\sqrt{\beta_P}}I \end{matrix}\right)
 \left(\begin{matrix} \mathbf{0} & \cdots & \mathbf{0} & -\sqrt{\beta_P} {A_P} & \frac{1}{\sqrt{\beta_P}}I\end{matrix}\right)
\end{aligned}
\end{equation*}

Then, for arbitrary $\omega=(x,\lambda)\neq 0$, we have

\begin{equation*}
\begin{aligned}
\omega^\top H \omega =\sum_{i=1}^p \left\|\frac{1}{\sqrt{\beta_i}} \lambda-\sqrt{\beta_i}A_ix_i \right\|^2+\sum_{i=1}^{p_1} \|x_i\|_{Q_i}^2 >0
\end {aligned}
\end{equation*}

and hence the matrix $H$ is positive definite.

\end{proof}

\begin{remark}
Based on the Lemma \ref{lem5.1} and Proposition \ref{pro5.1}, similar lemmas and theorems for convergence and convergence rate as those in Section 3 can be obtained for the splitting version of the penalty dual-primal ALM (\ref{5.1}). Hence it is easy to prove the worst-case $O(1/t)$ convergence rate of the spitting penalty dual-primal balanced ALM (\ref{5.1}), where $t$ denotes the total iteration counter. For brevity, we omit the details proof.
\end{remark}

\section{Numerical experiments}

In this section, two numerical tests and a practical application are used for the purpose of demonstrate the proposed algorithm performance. All code are written in Matlab and all experiments are performed in Matlab R2015b on a workstation with an Intel(R) Core(TM) i7-8550U CPU(1.80GHz) and 8GB RAM.

We firstly show the comparison among the penalty dual-primal balanced ALM (\ref{1.6}), the primal-dual balanced ALM proposed in \cite{14xu} and the balanced ALM proposed in \cite{12he} and  for soling the basic pursuit problem (equality-constrained $l_1$ minimization problem). The preliminary numerical result shows that the proposed method has a better performance.
\begin{example}
The basic pursuit problem is:
\begin{eqnarray}\label{6.1}
\min\{\|x\|_1\,|\,Ax=b, x\in \mathbb{R}^{n}\},
\end{eqnarray}
with $\|x\|_1=\sum_{i=1}^{n}|x_i|$ denotes the $l_1$-norm of a vector, $A\in \mathbb{R}^{m\times n}(m<n)$ is a given matrix and $b \in \mathbb{R}^{m}$ is a given vector.

Apply the proposed method (\ref{1.6}) to the problem (\ref{6.1}), the iterative scheme is as follows:

\begin{eqnarray*}
 \left\{ \begin{array}{ll}
\lambda^{k+1}= \lambda^{k}-\beta(Ax^{k}-b),\\
x^{k+1}= \arg \min  \left\{ \|x\|_1-\langle 2\lambda^{k+1}-\lambda^k,Ax-b \rangle
+\frac{1}{2}\|x-x^k\|_{\beta A^\top A +Q}^2  \right\}.
\end{array} \right.
\end{eqnarray*}

In particular, take $Q=\tau I-\beta A^\top A$ with $\tau >\beta \|A^\top A\|$, the iterate scheme could converted to:

\begin{eqnarray}\label{p6.2}
 \left\{ \begin{array}{ll}
\lambda^{k+1}= \lambda^{k}-\beta(Ax^{k}-b),\\
x^{k+1}= \arg \min  \left\{\|x\|_{1}+\frac{\tau}{2}\|x-x^k-\frac{1}{\tau}A^\top(2\lambda^{k+1}-\lambda^k)\|^2 \right\}.
\end{array} \right.
\end{eqnarray}

Then the solutions of the problem (\ref{p6.2}) are given respectively by the following explicit form:
\begin{eqnarray*}
 \left\{ \begin{array}{ll}
\lambda^{k+1}= \lambda^{k}-\beta(Ax^{k}-b),\\
x^{k+1}=S_{1/\tau}[x^{k}+\frac{1}{\tau}A^\top(2\lambda^{k+1}-\lambda^k)],
\end{array} \right.
\end{eqnarray*}
where $S_{\delta}(t)$ is the soft thresholding operator \cite{24che} defined as

\begin{eqnarray}\label{p6.3}
(S_{\delta}(t))_{i}:=(1-\delta/|t_{i}|)_{+}\cdot t_{i},i=1,2,\cdot\cdot\cdot,m.
\end{eqnarray}

Follow the same rules, apply the primal-dual balanced ALM  \cite{14xu} for (\ref{6.1}), the iterative scheme is read as:

\begin{eqnarray*}
{\rm (DP-BALM)} \left\{ \begin{array}{ll}
\lambda^{k+1}= \lambda^{k}- (\frac{1}{\tau}AA^\top+\delta I_m)^{-1} (Ax^{k}-b),\\
x^{k+1} = \arg \min  \left\{\|x\|_{1}+\frac{\tau}{2}\|x-x^k-\frac{1}{\tau}A^\top(2\lambda^{k+1}-\lambda^k)\|^2 \right\}.
\end{array} \right.
\end{eqnarray*}

And  the balanced ALM \cite{12he} for (\ref{6.1}) reads as :
\begin{eqnarray*}
{\rm (Balanced \ ALM)} \left\{ \begin{array}{ll}
x^{k+1} = \arg \min  \left\{\|x\|_{1}+\frac{\tau}{2}\|x-x^k-\frac{1}{\tau}A^\top\lambda^k\|^2 \right\},\\
\lambda^{k+1}= \lambda^{k}- (\frac{1}{\tau}AA^\top+\delta I_m)^{-1} (A(2x^{k+1}-x^{k})-b).
\end{array} \right.
\end{eqnarray*}

We generate the data by the same way in \cite{14xu}: the matrixes $A$ are generated from independently normal distribution $\mathcal{N}(0,1)$; for all tested algorithm the initial point $(x^0,\lambda^0)$ is randomly generated, and we take the following toned values of parameters for the mentioned experiments:
\begin{enumerate}
\item the penalty dual-primal balanced ALM: $\beta:=0.001$ and $\tau=2.5$;

\item the primal-dual balanced ALM and the balanced ALM :$\delta=1000$ and $\tau:2.5$.
\end{enumerate}

The termination criteria is defined by
\begin{eqnarray*}
R(k)=\max\left\{\|x^{k+1}-x^{k}\|,  \|\lambda^{k+1}-\lambda^{k}\| \right\} < 10^{-7}
\end{eqnarray*}

Table \ref{tab1} lists the  number of iterations and runtime in seconds respectively of the the PDP-ALM, the DP-BALM and the balanced ALM for solving the basic pursuit problem with different dimension $m \times n$ of $A$, and $CR=\|Ax-b\|^2$ stands for constrained residual. From the numerical experimental result, it is clearly that compared with the the PDP-ALM and the DP-BALM under different dimension of $A$, the proposed method has much better performs both in the number of iterations and runtime. To further visualize the numerical results, we also plot the convergence curves versus iteration numbers of some representative examples in Figure \ref{fig1}.

\begin{table}\label{tab1}
\centering\caption{The  number of iterations and runtime of the PL-ADMM and the proposed method for solving LASSO model }
\begin{tabular}{cccccccccc}\hline
 \multicolumn{1}{l}{}&\multicolumn{3}{l}{PDP-ALM}&\multicolumn{3}{l}{DP-ALM}&\multicolumn{3}{l}{B-ALM}\\
$m \times n$  &Iter. &Time &$CR$ &Iter. &Time &$CR$ &Iter. &Time &$CR$\\ \hline
$300\times500$	    &465	&0.11 	&2.86e-4	&562	&0.16 	&3.64e-4	&564	&0.18 	&3.64e-4\\
$400\times600$  	&503	&0.15 	&1.49e-4	&669	&0.28 	&1.82e-4	&671	&0.28 	&1.82e-4\\
$450\times750$  	&453	&0.18 	&6.60e-5	&610	&0.39 	&1.01e-4	&612	&0.35 	&1.01e-4\\
$500\times900$  	&373	&0.20 	&4.63e-5	&511	&0.39 	&5.98e-5	&513	&0.39 	&5.98e-5\\
$500\times1000$ 	&1294	&0.73 	&7.25e-5	&1969	&1.50 	&2.29e-4	&1971	&1.47 	&2.29e-4\\
$600\times1150$ 	&843	&0.70 	&1.77e-4	&1263	&1.40 	&1.65e-4	&1266	&1.34 	&1.70e-4\\
$700\times1300$ 	&672	&0.75 	&8.90e-5	&1073	&1.51 	&1.15e-4	&1074	&1.50 	&1.15e-4\\
$800\times1450$	    &455	&0.65 	&4.63e-5	&704	&1.26 	&1.53e-4	&706	&1.31 	&1.48e-4\\
$900\times1600$  	&547	&1.03 	&5.99e-5	&943	&2.21 	&8.24e-5	&945	&2.27 	&8.27e-5\\
$1000\times1750$	&876	&1.95 	&4.32e-5	&1523	&4.14 	&6.37e-5	&1525	&4.22 	&6.40e-5\\
$1100\times1900$	&725	&1.83 	&4.78e-5	&1179	&3.81 	&8.53e-5	&1181	&3.89 	&8.58e-5\\
$1100\times2000$	&694	&1.85 	&4.47e-5	&1263	&4.24 	&9.12e-5	&1265	&4.30 	&9.13e-5\\
$1200\times2150$	&406	&1.31 	&3.85e-5	&602	&3.65 	&5.38e-5	&896	&3.88 	&1.10e-4\\
$1300\times2300$	&420	&1.52 	&6.66e-5	&762	&3.53 	&6.27e-5	&760	&3.58 	&7.01e-5\\
\hline
\end{tabular}
\end{table}

\begin{figure}
\centering

\subfigure[$m \times n=300\times 500$]{
\includegraphics[width=70mm]{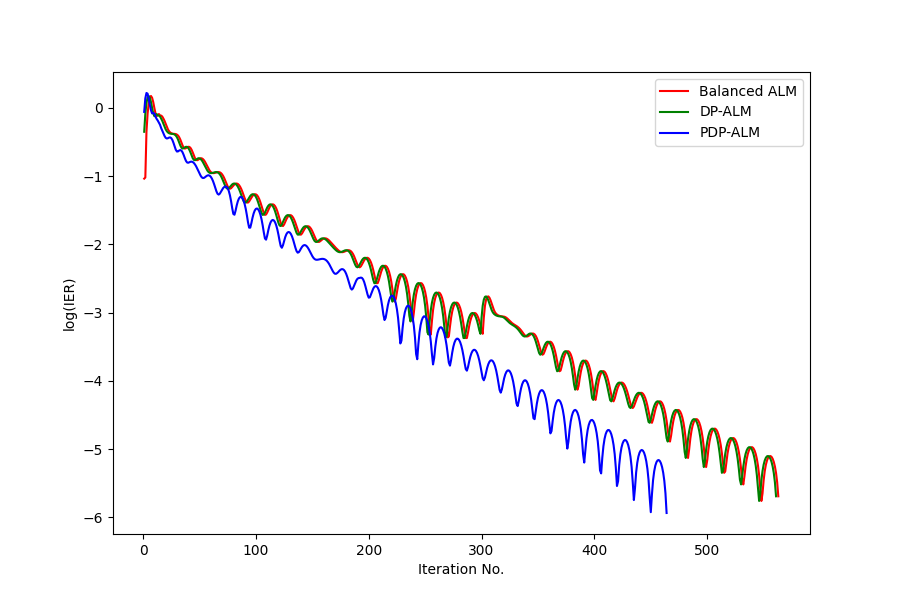}
}
\quad
\subfigure[$m \times n=450\times 750$]{
\includegraphics[width=70mm]{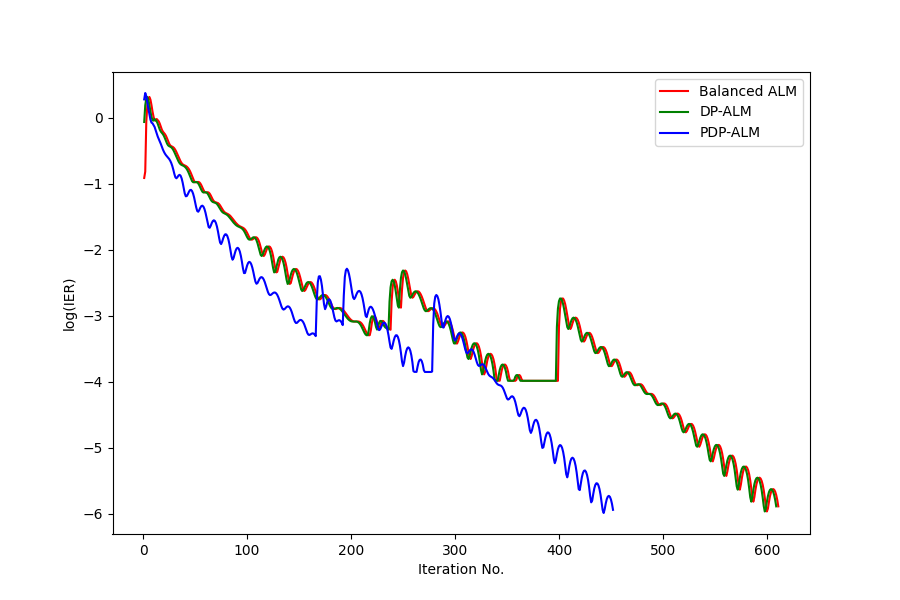}
}
\quad
\subfigure[$m \times n=500\times 900$]{
\includegraphics[width=70mm]{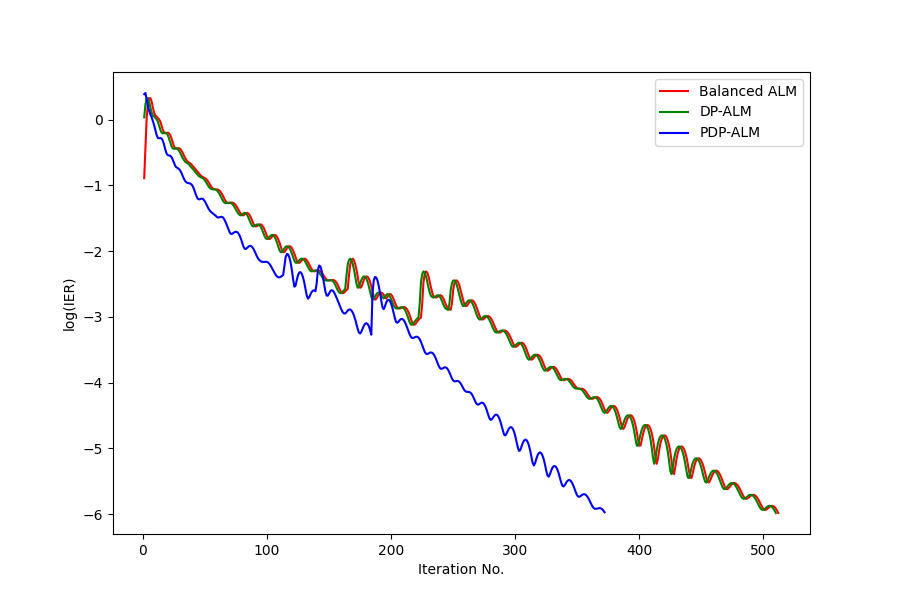}
}
\quad
\subfigure[$m \times n=700\times 1300$]{
\includegraphics[width=70mm]{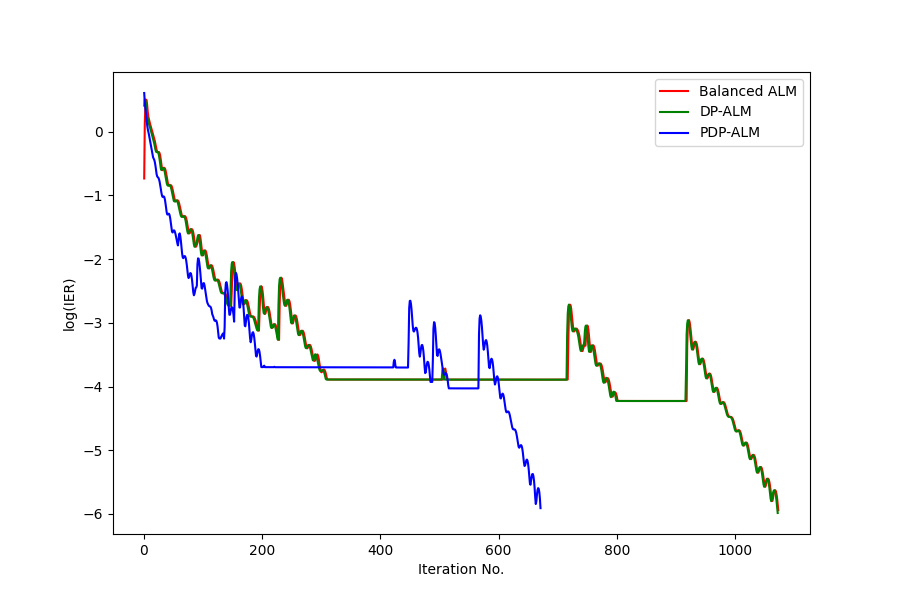}
}
\quad
\subfigure[$m \times n=800\times 1450$]{
\centering
\includegraphics[width=70mm]{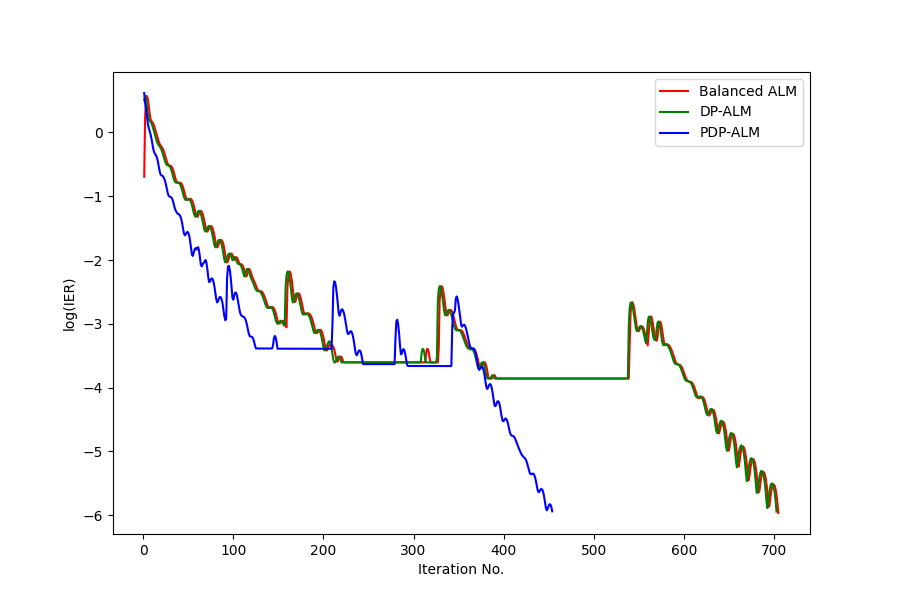}
}
\quad
\subfigure[$m \times n=1000\times 1750$]{
\includegraphics[width=70mm]{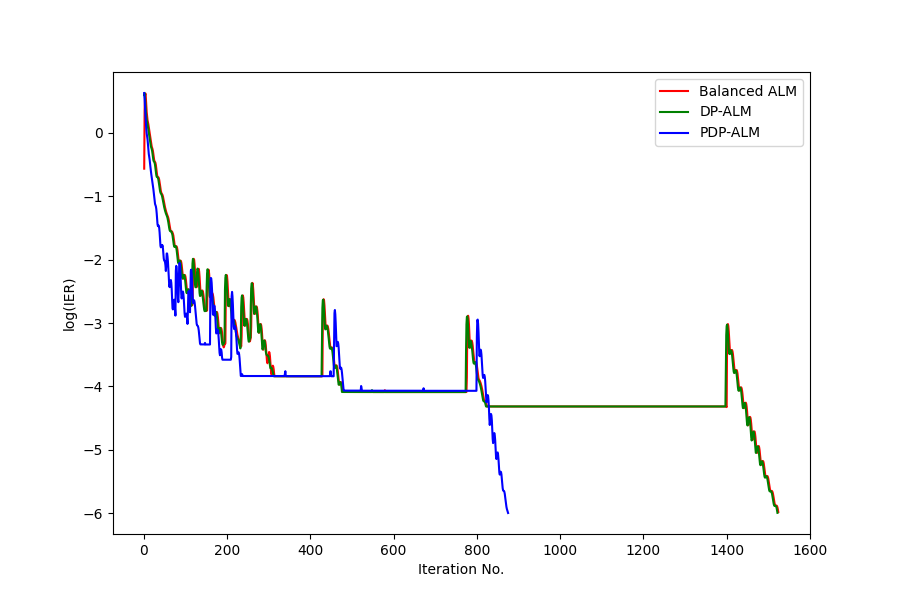}
}
\quad
\subfigure[$m \times n=1100\times 2000$]{
\centering
\includegraphics[width=70mm]{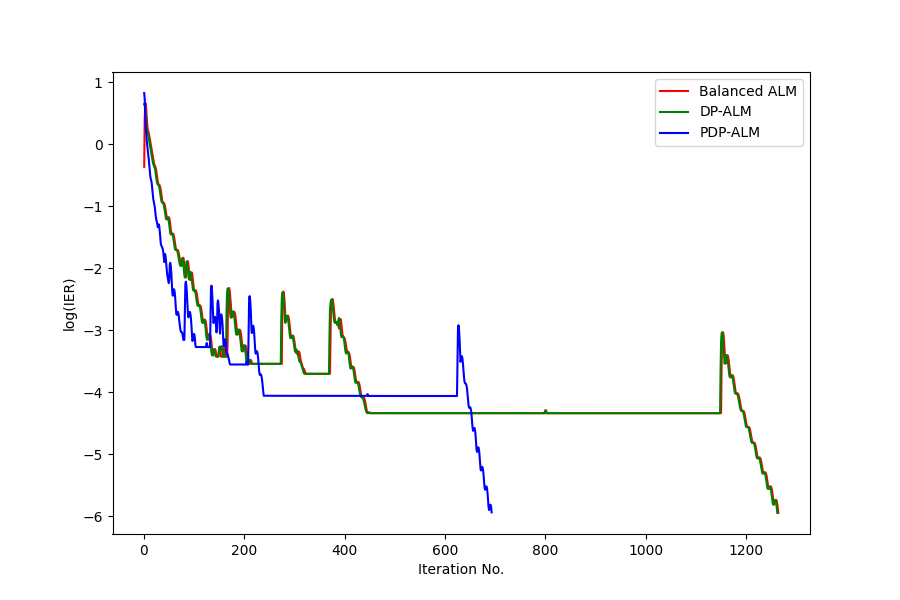}
}
\quad
\subfigure[$m \times n=1300\times 2300$]{
\centering
\includegraphics[width=70mm]{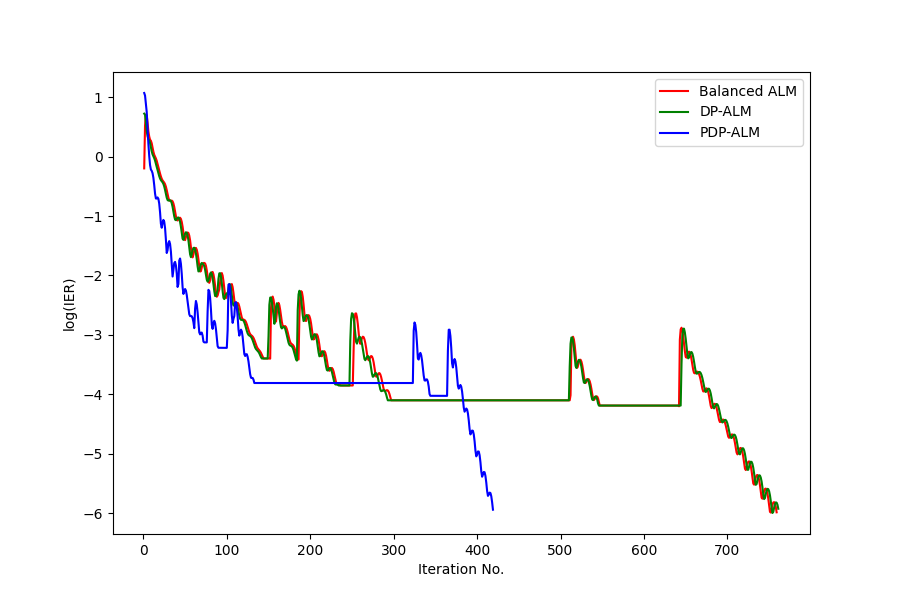}
}
\centering
\caption{Convergence curves of the DPD-ALM, DP-ALM and the balanced ALM compared with iteration number under various dimension of $A$}\label{fig1}
\end{figure}

\end{example}

Then we  show the comparison among the spitting penalty dual-primal balanced ALM (\ref{4.2}) , the PL-ADMM \cite{23yan} and the PIPL-ADMM \cite{24che} for solving  the well-known LASSO model, and the preliminary numerical result shows that the proposed method has a better performance.
\begin{example}
The LASSO model is:
\begin{eqnarray}\label{6.2}
\min_{y}  \frac{1}{2}\|Ay-b\|^{2}+\sigma\|y\|_{1},
\end{eqnarray}
where $\|y\|_{1}:=\Sigma^{n}_{i=1}|y_{i}|$, $A\in \mathbb{R}^{m\times n}$ is a design matrix usually with $m\ll n$, $m$ is the number of date point, $n$ is the number of features, $b\in \mathbb{R}^{m}$ is the response vector and $\sigma>0$ is a regularization parameter.

By a new auxiliary variable $x$, (\ref{6.2}) can be rewritten as the form:

\begin{eqnarray}\label{6.3}
\min\{\frac{1}{2}\|x-b\|^{2}+\sigma\|y\|_{1}\,|\,x-Ay=0, \ \ \ x\in \mathbb{R}^{m}, y\in \mathbb{R}^{n}\}.
\end{eqnarray}

which is a special case of (\ref{4.1}). Then apply the spitting penalty dual-primal balanced ALM (\ref{4.2}) to (\ref{6.3}),  we have:

\begin{eqnarray*}
 \left\{ \begin{array}{ll}
\lambda^{k+1}= \lambda^{k}-\beta_{1}(x-Ay)-\beta_{2}(x-Ay),\\
x^{k+1}= \arg \min  \left\{ \frac{1}{2}\|x-b\|^2-\langle 2\lambda^{k+1}-\lambda^k,x-b \rangle
+\frac{1}{2}\|x-x^k\|_{\beta_1 I^\top I +Q_{1}}^2  \right\}   \\
y^{k+1}= \arg \min  \left\{ \sigma\|y\|_{1}-\langle 2\lambda^{k+1}-\lambda^k,-Ay-b \rangle
+\frac{1}{2}\|y-y^k\|_{\beta_2 A^\top A +Q_{2}}^2  \right\}.
\end{array} \right.
\end{eqnarray*}

In particular, take $Q_1=\tau_1 I-\beta_1 I$ with $\tau_1 >\beta_1 \|I^\top I\|$, $Q_2=\tau_2 I-\beta_2 A^\top A$ with $\tau_2 >\beta_2 \|A^\top A\|$, the iterate could converted to:

\begin{eqnarray}
 \left\{ \begin{array}{ll}{\label{6.4}}
\lambda^{k+1}= \lambda^{k}-\beta_{1}(x-Ay)-\beta_{2}(x-Ay),\\
x^{k+1}= \arg \min  \left\{ \frac{1}{2}\|x-b\|^2+\frac{\tau_1}{2}\|x-x^k-\frac{1}{\tau_1}(2\lambda^{k+1}-\lambda^k)\|^2  \right\} \\
y^{k+1}= \arg \min  \left\{ \sigma\|y\|_{1}+\frac{\tau_2}{2}\|y-y^k-\frac{1}{\tau_2}A^\top(2\lambda^{k+1}-\lambda^k)\|^2 \right\}.
\end{array} \right.
\end{eqnarray}

\begin{table}[htbp]
\centering\caption{The  number of iterations and runtime of the PL-ADMM and the proposed method for solving LASSO model }\label{tab2}
\begin{tabular}{ccccccccccc}\hline
 \multicolumn{1}{l}{}&\multicolumn{2}{l}{PL-ADMM}&\multicolumn{2}{l}{PIPL-ADMM}&\multicolumn{2}{l}{PDP-ALM}\\
$\gamma$  &Iter.1&Time.1&Iter.2&Time.2&Iter.3&Time.3& $\frac{Iter.3}{Iter.1}$& $\frac{Time3}{Time1}$& $\frac{Iter.3}{Iter.2}$& $\frac{Time3}{Time2}$\\\hline
$0.05$   &376	&29.35	&204	&16.16  &101	&3.53	&0.27	&0.12 &0.50 	&0.22\\
$0.10$	&380	&29.63	&417	&33.98  &111	&3.92	&0.29	&0.13 &0.27 	&0.12\\
$0.15$	&384	&30.13	&397	&31.66  &118	&4.15	&0.31	&0.14 &0.30 	&0.13\\
$0.20$	&399	&32.97	&410	&34.58  &123	&4.34	&0.31	&0.13 &0.30 	&0.13\\
$0.25$	&413	&32.18	&409	&33.57  &127	&4.54	&0.31	&0.14 &0.31 	&0.14\\
$0.30$   &409	&31.92	&406	&34.08  &130	&4.56	&0.32	&0.14 &0.32 	&0.13\\
$0.35$   &420	&33.02	&404	&32.64  &132	&4.65	&0.31	&0.14 &0.33 	&0.14\\
$0.40$   &420	&32.87	&397	&32.32  &134	&4.76	&0.32	&0.14 &0.34 	&0.15\\
$0.45$   &431	&33.82	&392	&31.34  &136	&4.78	&0.32	&0.14 &0.35 	&0.15\\
$0.50$   &438	&34.67	&392	&35.86  &137	&4.83	&0.31	&0.14 &0.35 	&0.13\\
$0.55$   &443	&34.69	&395	&31.34  &138	&4.85	&0.31	&0.14 &0.35 	&0.15\\
$0.60$   &452	&35.34	&395	&31.78  &138	&4.87	&0.31	&0.14 &0.35 	&0.15\\
$0.65$   &407	&23.41	&394	&31.38  &138	&3.53	&0.34	&0.15 &0.35 	&0.11\\
$0.70$   &413	&23.74	&400	&32.67  &137	&3.53	&0.33	&0.15 &0.34 	&0.11\\
\hline
\end{tabular}
\end{table}

\begin{figure}[htbp]
\begin{minipage}[t]{7cm}
\centering
\mbox{\resizebox{!}{65mm}{\includegraphics{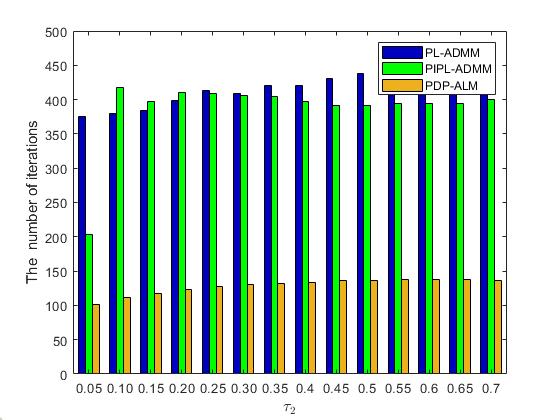}}}\\
\end{minipage}
\hspace{1cm}
\begin{minipage}[t]{7cm}
\centering
 \mbox{\resizebox{!}{65mm}{\includegraphics{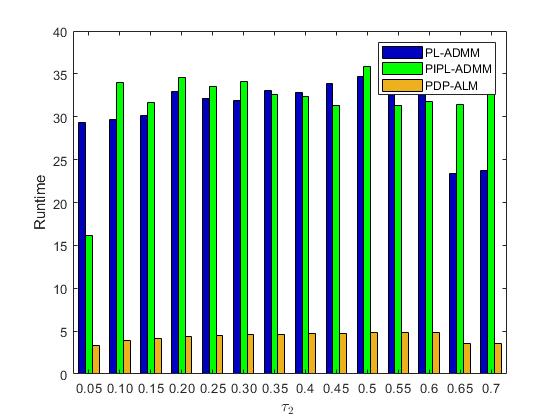}}}\\
\end{minipage}
\caption{The number of iterations and average runtime of splitting PDP-ALM and PL-ADMM  in different $\tau_2$}\label{fig2}
\end{figure}

Then the solutions of the problem (\ref{6.4}) are given respectively by the following explicit form:

\begin{eqnarray*}
 \left\{ \begin{array}{ll}
\lambda^{k+1}= \lambda^{k}-\beta_{1}(x-Ay)-\beta_{2}(x-Ay),\\
x^{k+1}= \frac{1}{\tau_1} [ \tau_1x^k+(2\lambda^{k+1}-\lambda^k)  ] \\
y^{k+1}=S_{\sigma/\tau_2}[y^{k}+\frac{1}{\tau_2}A^\top(2\lambda^{k+1}-\lambda^k)],
\end{array} \right.
\end{eqnarray*}
where $S_{\delta}(t)$ is the soft thresholding operator \cite{24che} defined as (\ref{p6.3}).

We generate the data by the same way in \cite{24che}: we first choose $A_{ij}\sim\mathcal{N}(0,1)$ and then scaled the columns to have unit norm. We use the script \lq sprandn\rq \, to generate a sparse vector $y^{*}$ which have approximately density $=100/n$ non-zeros entries taken from the normal distribution with zero mean and unit variance. We generate $b$ via $b:=Ay^{*}+e$, where $e$ is a small white noise taken from $e\sim\mathcal{N}(0,10^{-3}I)$. We choose the dimension of $A$ is $1050\times3500$. We set the regularization parameter $\sigma$ to $0.1$, and for all tested algorithm the initial point $(x^0,y^0,\lambda^0)$ is randomly generated, and we take the following toned values of parameters for the mentioned experiments:
\begin{enumerate}
\item The spitting penalty dual-primal balanced ALM: $\beta_1 =\beta_2:=\frac{2-\tau_2}{\tau_2|\tau_2-1|}$ and $\tau_1 :=\frac{|\tau_2-1|}{5\beta_1\tau_2}+\frac{4}{5}$;

\item PL-ADMML:$\beta=\frac{2-\tau_2}{\tau_2|\tau_2-1|}$ and $\tau:=\frac{|\tau_2-1|}{5\beta_1\tau_2}+\frac{4}{5}$,
\end{enumerate}
and we all set $\tau_2=0.05, 0.1, 0.15, 0.2, 0.25, 0.3, 0.35, 0.4, 0.45, 0.5, 0.55, 0.6, 0.65, 0.7. $

 The termination criteria is defined by
\begin{eqnarray*}
\max\left\{\|x^{k+1}-x^{k}\|, \|y^{k+1}-y^{k}\|, \|\lambda^{k+1}-\lambda^{k}\| \right\} < 10^{-10}
\end{eqnarray*}

Table \ref{tab2} lists the  number of iterations and runtime in seconds respectively of the the PL-ADMM and the spitting penalty dual-primal balanced ALM for solving the LASSO model with different parameter $\tau_2$. From the numerical experimental result, it is clearly that compared with the PL-ADMM under different parameter $\tau_2$, the proposed method has much better performs both in the number of iterations and runtime. To further visualize the numerical results, we also plot the iterations results in terms of the various parameters $\tau_2$ in Figure \ref{fig2}.
\end{example}

\section{Conclusions}

 This paper have proposed a penalty dual-primal augmented lagrangian method for solving convex minimization problems under linear equality or inequality constraints, and two extensions to solve the multiple-block separable convex programming problems. The global convergence and sub-linear convergence rate of the proposed methods have been establish in the lens of variational analysis. Furthermore, the numerical test on the basic pursuit problem and the lasso model demonstrate that the proposed method has better performance compared with the dual-primal ALM and the balanced ALM. This work may enhance the rich literature of the most recent balanced ALM.

\section*{ Acknowledgments.}

This paper was partially supported by  the Youth Project of Science and Technology Research Program of Chongqing Education Commission of China (No. KJQN202201802), the
the Natural Science Foundation of China (No. 12071379),
 the Natural Science Foundation of Chongqing (No. cstc2021jcyj-msxmX0925, cstc2022ycjh-bgzxm0097).

\end{document}